\documentclass[12pt,reqno]{amsart}
\usepackage{amsmath, amsthm, amssymb}
\usepackage{hyperref}

\usepackage{mathtools}
\DeclarePairedDelimiter{\ceil}{\lceil}{\rceil}

\advance \topmargin by -\headheight
\advance \topmargin by -\headsep
     
\evensidemargin \oddsidemargin
\newtheorem{thm}{Theorem}[section]
\newtheorem{lemma}[thm]{Lemma}

\newtheorem{cor}[thm]{Corollary}

\theoremstyle{definition}
\newtheorem{defn}[thm]{Definition}

\theoremstyle{remark}

\numberwithin{equation}{section}

\newcommand*\wrapletters[1]{\wr@pletters#1\@nil}
\def\wr@pletters#1#2\@nil{#1\allowbreak\if&#2&\else\wr@pletters#2\@nil\fi}

\def\alp{{\alpha}}

\def\eps{\varepsilon}

\def\le{\leqslant} \def\ge{\geqslant}

\def \bR {\mathbb R}
\def \bZ {\mathbb Z}

\def \ba {\mathbf a}
\def \bb {\mathbf b}
\def \bc {\mathbf c}
\def \bd {\mathbf d}
\def \bt {\mathbf t}
\def \bu {\mathbf u}
\def \bv {\mathbf v}
\def \bx {\mathbf x}
\def \bw {\mathbf w}
\def \by {\mathbf y}
\def \bz {\mathbf z}

\def \bzero {\mathbf 0}
\def \blambda {\boldsymbol{\lambda}}

\def \cE {\mathcal E}
\def \cF {\mathcal F}

\def \hw {\hat{w}}

\begin{document}
\title[Cubic diophantine inequalities for split forms]{Cubic diophantine inequalities for split forms}
\author[Sam Chow]{Sam Chow}
\address{School of Mathematics, University of Bristol, University Walk, Clifton, Bristol BS8 1TW, United Kingdom}
\email{Sam.Chow@bristol.ac.uk}
\thanks{}
\date{}
\begin{abstract} 
Denote by $s_0^{(r)}$ the least integer such that if $s \ge s_0^{(r)}$, and $F$ is a cubic form with real coefficients in $s$ variables that splits into $r$ parts, then $F$ takes arbitrarily small values at nonzero integral points. We bound $s_0^{(r)}$ for $r \le 6$.
\end{abstract}
\maketitle

\section{Introduction}
\label{intro}
Let $F$ be an indefinite homogeneous polynomial of degree $d \ge 2$ in $s$ variables, with real coefficients. The form $F$ \emph{takes small values} if there exists $\bx \in \bZ^s \setminus \{ \bzero \}$ such that
\begin{equation} \label{de}
| F(\bx) | < 1.
\end{equation}
Schmidt \cite{Sch1980} showed that if $d$ is odd then there exists $s_0 = s_0(d)$ such that if $s \ge s_0$ then $F$ takes small values.

Henceforth, let $s_0$ be the least integer with the above property, whenever such an integer exists. Naturally we seek upper bounds for $s_0$. For cubic forms, Freeman \cite{Fre2000} holds the record $s_0(3) \le 359~551~882$, improving significantly on previous work of Pitman \cite{Pit1968}. Contrastingly, Baker, Br\"udern and Wooley \cite{BBW1995} have shown that any additive cubic form in seven variables takes small values (Baker \cite{Bak1989} had already shown this for integral forms).

We consider an intermediate case where only some additive structure is present, a game recently played in a different context by Browning \cite{Bro2010} and continued by Dai and Xue \cite{DX2013}. Let $r$ be a positive integer. A cubic form $F$ in $s$ variables \emph{splits into $r$ parts} if there exist positive integers $a_1,\ldots,a_r$ and nonzero cubic forms $C_1, \ldots, C_r$ such that $a_1+ \ldots+a_r=s$ and
\[ F(x_1,\ldots, x_s) = \sum_{j=1}^r C_j(x_{A_{j-1}+1}, \ldots, x_{A_j}), \]
where $A_0 = 0$ and $A_j = a_1 + \ldots + a_j$ for $j=1,2,\ldots, s$. Denote by $s_0^{(r)}$ the least integer such that if $F$ is a cubic form that splits into $r$ parts and $s \ge s_0$ then $F$ takes small values.

\begin{thm} \label{main} 
\begin{align*} &s_0^{(1)} \le 358~823~708, & s_0^{(2)} \le 120~897~257,\qquad & s_0^{(3)} \le 35~042~291, \\
&s_0^{(4)}  \le 8~324~100,  &s_0^{(5)}  \le 1~164~774,\qquad &  s_0^{(6)}  \le 77~027.
\end{align*}
\end{thm}

Note that if $s \ge s_0^{(r)}$ and $\eta > 0$ then the inequality
\[ |F(\bx)| < \eta \]
has a nontrivial solution for any cubic form $F$ that splits into $r$ parts, since the form $\eta^{-1}F$ takes small values. The improvement in $s_0^{(1)} = s_0(3)$ is obtained via two suggestions made in the introduction of \cite{Fre2000}. Though this improvement is minor, we will justify it, since the ideas will also improve our other bounds and perhaps future related results. If one assumed that the parts were of roughly equal dimensions $a_1, \ldots, a_r$, then much better bounds could be obtained. For instance our methods easily show that if $F$ splits into three parts of dimension at least $270~187$ then $F$ takes small values.

We briefly discuss the case that $F$ has integer coefficients, wherein the inequality \eqref{de} reduces to $F(\bx) = 0$. Heath-Brown \cite{HB2007} has shown that 14 variables suffice to ensure that a cubic form has a nontrivial zero, improving on Davenport's already spectacular previous record \cite{Dav1963} of 16 variables. Under the additional premise that $F$ is nonsingular, Heath-Brown \cite{HB1983} proved that only ten variables are necessary. This has been sharpened by Hooley \cite{Hoo1988, Hoo1991, Hoo1994, Hoo2013}, who established the Hasse principle for nonary cubic forms defining hypersurfaces with at most ordinary double points.

For quadratic forms we have Margulis' celebrated proof of the Oppenheim conjecture (see \cite{Mar1997}). This states that the values taken at integral points by an indefinite quadratic form in at least three variables, whose coefficients are not all in rational ratio, are dense on the real line. Margulis' result shows that $s_0(2) \le 5$, since Meyer demonstrated that indefinite quadratic forms with integer coefficients in at least five variables represent zero nontrivially (see \cite{Ser1973}). In fact $s_0(2) = 5$, for if $p \equiv 3 \mod 4$ is a prime then
\[ a^2 + b^2 - p(c^2 + d^2) = 0 \]
has no nontrivial integer solutions.

Now we summarise some of Freeman and Wooley's work on additive forms (see \cite{Woo2003}). Let $\cF(d)$ denote the least integer $t$ such that any indefinite diagonal form of degree $d$ with real coefficients in at least $t$ variables takes small values. Then 
\[ \cF(4) \le 12, \qquad \cF(5) \le 18 \]
and
\[ \cF(d) \le d (\log d + \log \log d + 2 + o(1)).\]
Our discussion up to this point implies that $\cF(2) = 5$ and $\cF(3) \le 7$.

Much less is known about diophantine inequalities for general forms. The only successful approach thus far has been to `quasi-diagonalise' and then use results about additive forms. We shall also follow this pattern. If a cubic form $F$ splits into $r$ parts, we will use on each part a quasi-diagonalisation procedure due to Freeman \cite{Fre2000}, thereby approximating $F$ by a diagonal form in some subspace. The resulting error term necessitates that we find a small solution to a diophantine inequality for additive cubic forms, and for this we use the work of Br\"udern \cite{Bru1996}.

This paper is organised as follows. In \S \ref{prelim}, we slightly modify some results of Freeman \cite{Fre2000}, thereby obtaining the bound $s_0^{(1)} \le 358~823~708$. We then recall two results of Br\"udern \cite{Bru1996}. In \S \ref{strat}, we elaborate on our strategy for deducing the remaining bounds stated in Theorem \ref{main}. We quasi-diagonalise using different exponents for different parts. In \S \ref{finish} we implement our strategy.

Bold face will be used for vectors, for instance we shall abbreviate $(u_1, \ldots, u_n)$ to $\bu$ and write $|\bu| = \max |u_i|$. For a form $F$, denote by $|F|$ the maximum of the absolute values of its coefficients.

The author thanks his supervisor Trevor Wooley for suggesting this problem, as well as for his continued support and encouragement.

\section{Preliminary manoeuvres}
\label{prelim}
In \cite{Fre2000}, ``a trick used by Pitman \cite{Pit1968}" is asserted to give $s_0(3) \le 359~547~172$. Our treatment of this more closely resembles \cite[p. 10]{HB2010}. Note that a finite number of integral vectors are linearly dependent over $\bZ$ if and only if they are linearly dependent over $\bR$. 

\begin{lemma} \label{trick}
Let $F$ be a cubic form with real coefficients in $s$ variables, let $E$ be a positive real number, and let $n\ge 2$ be an integer. Let $N > 0$ be sufficiently large in terms of $E$ and $F$. Suppose there exist linearly dependent nonzero vectors $\bx_1,\ldots, \bx_n \in \bZ^s$ such that for any $\bu \in \bR^n$ we have
\begin{equation} \label{quasi}
F(u_1 \bx_1 + \ldots + u_n\bx_n) = \sum_{i \le n}F(\bx_i)u_i^3 + O(N^{-E} |\bu|^3),
\end{equation}
where the implicit constant may depend on $F$. Then $F$ takes small values.
\end{lemma}

\begin{proof}
There exists $\bc \in \bZ^n - \{ \bzero \}$ such that
\[ c_1 \bx_1 = c_2 \bx_2 + \ldots + c_n\bx_n, \]
and without loss of generality $c_2 \ge |c_i|$ for $i=1,2,\ldots,n$. Specialising $\bu = (c_1, c_2, 0, \ldots, 0)$ in \eqref{quasi} yields
\[
F(c_1 \bx_1 + c_2 \bx_2) =  c_1^3 F(\bx_1)+  c_2^3 F(\bx_2)+ O(N^{-E}  c_2^3),
\]
and we also have
\begin{align*}
F(c_1 \bx_1 + c_2 \bx_2) &= F(2c_2 \bx_2 + c_3 \bx_3 + \ldots + c_n \bx_n) \\
&= 8c_2^3 F(\bx_2) + \sum_{i=3}^n c_i^i F(\bx_i) + O(N^{-E} c_2^3),
\end{align*}
so
\begin{equation} \label{three}
c_1^3 F(\bx_1) = 7c_2^3 F(\bx_2) + \sum_{i=3}^n c_i^3 F(\bx_i) + O(N^{-E} c_2^3).
\end{equation}
Moreover,
\begin{align}
\notag c_1^3 F(\bx_1) &= F(c_1 \bx_1) = F(c_2 \bx_2+ \ldots +c_n \bx_n)  \\
\label{four} &= \sum_{i=2}^n c_i^3 F(\bx_i) + O(N^{-E} c_2^3).
\end{align}
Equations \eqref{three} and \eqref{four} now yield
\[ F(\bx_2) \ll N^{-E}.\]
Since $N$ is large we conclude that $|F(\bx_2)| < 1$, so $F$ takes small values.
\end{proof}

The following is the same as \cite[Definition 2]{Fre2000}, except we do not insist that the quasi-diagonalising vectors be linearly independent.
\begin{defn}
Let $n$ be a positive integer and $E$ a positive real number. Let $\hw_3^{(n)}(E)$ be the least positive integer $t$ such that if $F$ is a form in more than $t$ variables and $N$ is sufficiently large in terms of $s$ and $E$, then there exist $\bx_1, \ldots, \bx_n \in \bZ^s \setminus \{ \bzero \}$ with $|\bx_j| \le N$ for $j=1,2,\ldots,n$ such that if $\bu \in \bR^n$ then 
\[
F(u_1 \bx_1 + \ldots + u_n\bx_n) = \sum_{i \le n}F(\bx_i)u_i^3 + O(N^{-E} |F| \cdot |\bu|^3).
\]
\end{defn}

\begin{lemma} Let $0 < \delta < 1$, let $E_1, E_2$ and $E_3$ be positive real numbers, and let $n$ be a positive integer. Put
\[ E = \min(E_1 \delta + \delta - 3, E_2 -E_2 \delta - 3 \delta, E_3 -2), \qquad M= \hw_3^{(n)}(E_2) \]
and 
\[ s = 1 + w_1^{(0)}(n(n+1)/2, E_3), \]
where $w_1^{(\cdot)}(\cdot,\cdot)$ is the positive integer defined in \cite[Definition 1]{Fre2000}. Assume that $E >0$. Then
\[ \hw_3^{(n+1)}(E) \le \max(s-1, w_1^{(M)}( s(s+1)/2, E_1)). \]
\end{lemma}

\begin{proof} We follow, \emph{mutatis mutandis}, the proof of \cite[Lemma 3]{Fre2000}. Let 
\[ B > \max(s-1, w_1^{(M)}( s(s+1)/2, E_1)), \]
let $F$ be a cubic form with real coefficients in $B$ variables, and let $N > 0$ be large. Again $T$ is a subspace of $\bR^B$, but now $\bx$ is free to be any vector in $\bR^B$, and $\bd_1,\ldots,\bd_{M+1}$ are not restricted to lie in $U$ either. The vectors $\ba_1, \ldots, \ba_n \in \bZ^{M+1} \setminus \{0\}$ are no longer necessarily linearly independent. The independence of the vectors $\bd_1, \ldots, \bd_{M+1}$ nonetheless ensures that $R$ is a form in $M+1$ variables, and also that the vectors $\bb_i$ are nonzero. The rest of the proof is identical to that of \cite[Lemma 3]{Fre2000}.
\end{proof}

The next corollary now follows in the same way as \cite[Corollary 2]{Fre2000}.

\begin{cor} Let $0< \delta < 1$, let $E$ be a positive real number, let $n$ be a positive integer, and put
\[ s = 1 + \ceil{(E+3)n(n+1)/2}.\]
Then
\[ \hw_3^{(n+1)}(E) \le \ceil{s(s+1)(E+3)/(2\delta)} + \hw_3^{(n)} ((E+3\delta) / (1-\delta)). \]
\end{cor}

We can now follow the proof of \cite[Theorem 1]{Fre2000} on \cite[p. 34]{Fre2000}, for if $\bx_1,\ldots, \bx_9$ were linearly dependent then $F$ would take small values by Lemma \ref{trick}. This leads to the following observation.

\begin{lemma} \label{obs9}
Let $\eps$ be a positive real number, and let $F$ be a real cubic form in $s$ variables. Put $n=9$, write $E = 24+\eps$, and let $N>0$ be sufficiently large in terms of $E$ and $|F|$. Suppose there exist $\bx_1,\ldots, \bx_9 \in \bZ^s \setminus \{\bzero \}$ with $|\bx_i| \le N$ for $i=1,2,\ldots, 9$ such that \eqref{quasi} holds for any $\bu \in \bR^9$. Then $F$ takes small values.
\end{lemma}

Using Freeman's choice of parameters on \cite[p. 35]{Fre2000} and $\eps = 0.00001$ yields $\hw_3^{(9)}(24+\eps) \le 359~547~171$ which, by Lemma \ref{obs9}, recovers $s_0(3) \le 359~547~172$, as claimed in \cite{Fre2000}. We shall optimise parameters using the \emph{Microsoft Excel} `Solver', adjusting the options to ensure greater precision. The choices \wrapletters{0.12192811498}7,~\wrapletters{0.139022584630}1,~\wrapletters{0.161610985231}5,~
\wrapletters{0.193206368538}7,\wrapletters{~0.23845892916}1,~\wrapletters{0.318020672526}4,
~\wrapletters{0.4759560851055} and \wrapletters{0.99999999999} for $\delta$, and $\eps = 10^{-13}$, yield $\hw_3^{(9)}(24+\eps) \le 358~823~707$ which, in light of Lemma \ref{obs9}, yields
\[ s_0^{(1)} = s_0(3) \le 358~823~708. \]
We use this method to obtain upper bounds for general $\hw_3^{(n)}(E)$.

We will need the following results of Br\"udern.

\begin{thm} \cite[p. 2]{Bru1996} \label{Bru9}
Let $\theta > 0$, and let $\blambda \in \bR^9$ with $|\lambda_1|,\ldots, |\lambda_9| \ge1$. Then there exists a solution $\bt \in \bZ^9$ to the system
\[ |\lambda_1 t_1^3+ \ldots + \lambda_9 t_9^3| < 1, \qquad 0 < \sum_{i\le9} |\lambda_i t_i^3| \ll |\lambda_1 \cdots \lambda_9|^{1+\theta}.\]
\end{thm}

\begin{thm} \cite[p. 1]{Bru1996} \label{Bru8}
Let $\theta > 0$, and let $\blambda \in \bR^8$ with $|\lambda_1|,\ldots, |\lambda_8| \ge1$. Then there exists a solution $\bt \in \bZ^8$ to the system
\[ |\lambda_1 t_1^3+ \ldots + \lambda_8 t_8^3| < 1, \qquad 0 < \sum_{i\le8} |\lambda_i t_i^3| \ll |\lambda_1 \cdots \lambda_8|^{15/8+\theta}.\]
\end{thm}

\section{Strategy}
\label{strat}

For the remainder of this paper, let $F$ be a cubic form with real coefficients in $s$ variables that splits into $r$ parts of dimensions $a_1 \le \ldots \le a_r$, put $\eps = 10^{-13}$, and let $N$ denote a large positive number. Implicit constants in Vinogradov and Landau notation may henceforth depend on $F$.

We expect good bounds if many of our parts are large, since we can then quasi-diagonalise each large part. For example, if $r \ge 3$ and $a_1 > \hw_3^{(3)}(24+\eps)$ then $F$ takes small values, by Lemma \ref{obs9} (note that $\hw_3^{(3)}(24+\eps) \le 270~186$). Consequently, in proving Theorem \ref{main}, we need a method that is effective in the case that $a_1, \ldots, a_{r-2}$ are small. We will either quasi-diagonalise the largest part or the largest two parts. 

\textbf{Case: $r=2,3,4$.} Here we use nine quasi-diagonalising vectors.

\begin{lemma} \label{first9}
Suppose $a_{r-1} < (s-\hw_3^{(10-r)}(27-3r+\eps)) / (r-1)$. Then $F$ takes small values.
\end{lemma}

\begin{proof}
Now $a_r \ge s-(r-1)a_{r-1} > \hw_3^{(10-r)}(E)$, where $E= 27-3r+\eps$. By letting $\bx_i$ be the $(A_{i-1}+1)$st standard basis vector for $i=1,2,\ldots, r-1$ and choosing $\bx_r,\ldots,\bx_9 \in \bR^s \setminus \{ \bzero \}$ using the definition of $\hw_3^{(10-r)}(E)$, we deduce for all $\bu \in \bR^9$ that
\begin{equation} \label{first9diag}
F(u_1 \bx_1 + \ldots + u_9\bx_9) = \sum_{i \le 9}F(\bx_i)u_i^3 + O(N^{-E} c),
\end{equation}
where $c = \max(|\bu_r|, \ldots, |\bu_9|)^3$. Note that for $i=r,\ldots,9$ we have $|\bx_i| \le N$, and so $F(\bx_i) \ll N^3$. We may assume that $|F(\bx_1)|, \ldots, |F(\bx_9)| \ge 1$, since otherwise $F$ takes small values. Let $\theta$ be sufficiently small compared to $\eps = 10^{-13}$. By Theorem \ref{Bru9} we may choose $\bu \in \bR^9 \setminus \{\bzero\}$ such that
\[
\Bigl| \sum_{i \le n}2F(\bx_i)u_i^3 \Bigr| < 1
\]
and $c \ll (N^3)^{9-r+\theta}$, giving $|F(u_1 \bx_1 + \ldots + u_9\bx_9)| < 1$. This implies that $F$ takes small values, for if $\bx_1, \ldots, \bx_9$ were linearly dependent then by \eqref{first9diag} we could apply Lemma \ref{trick}.
\end{proof}

We wish to show that $F$ takes small values. By Lemma \ref{first9} we may assume that
\begin{equation} \label{initial9} a_{r-1} \ge (s-\hw_3^{(10-r)}(27-3r+\eps)) / (r-1). \end{equation}
In this case we might quasi-diagonalise the largest two parts.

\begin{lemma} \label{second9} 
Let $E_1 > 3$ and $E_2$ be real numbers satisfying 
\[ (E_1-3)(E_2-3(8-r)) > 18(9-r). \]
Suppose $a_{r-1} > \hw_3^{(2)}(E_1)$ and $a_r > \hw_3^{(9-r)}(E_2)$. Then $F$ takes small values.
\end{lemma}

\begin{proof}
Since $E_1 > 3 +18(9-r)/(E_2 - 3(8-r))$ we may choose $\alpha > 6/(E_2 - 3(8-r))$ such that $E_1 > 3+3\alpha(9-r)$. Let $\bw_i$ be the $(A_{i-1}+1)$st standard basis vector for $i=1,2,\ldots, r-2$. There exist $\bx_1, \bx_2, \by_1, \ldots, \by_{9-r}$ such that $|\bx_j| \le N$ ($j=1,2$), $|\by_k| \le N^\alp$ ($k=1,2,\ldots,9-r$) and for any $\bt \in \bR^{r-2}, \bu \in \bR^2$ and $\bv \in \bR^{9-r}$ we have
\begin{align} \notag
F(\bz) = \sum_{i \le r-2} F(\bw_i)t_i^3 + \sum_{j \le 2} F(\bx_j)u_j^3 + 
\sum_{k \le 9-r} F(\by_k)v_k^3&
\\ \label{error} + O(N^{-E_1} |\bu|^3 + N^{-\alp E_2}|\bv|^3),&
\end{align}
where $\bz = t_1\bw_1+ \ldots + t_{r-2} \bw_{r-2}+ u_1 \bx_1+ u_2 \bx_2+v_1 \by_1 + \ldots + v_{9-r}\by_{9-r}$.

Let $\theta > 0$ be small in terms of $E_1,E_2$ and $\alpha$. We may assume that $|F(\bw_i)|, |F(\bx_j)|, |F(\by_k)| \ge 1$, since otherwise $F$ takes small values. For $j=1,2$ and $k=1,2,\ldots,9-r$ we have $F(\bx_i) \ll N^3$ and $F(\by_j) \ll N^{3\alp}$. By Theorem \ref{Bru9} we may choose nonzero vectors $\bt \in \bR^{r-2}$, $\bu \in \bR^2$ and $\bv \in \bR^{9-r}$ satisfying
\begin{equation} \label{bound}
|\bu|^3 \ll
N^3(N^{3\alpha})^{9-r} N^\theta, 
\qquad |\bv|^3 \ll 
(N^3)^2 (N^{3\alpha})^{8-r} N^\theta
\end{equation}
and
\begin{equation} \label{small}
\Bigl| \sum_{i \le r-2} 2F(\bw_i)t_i^3 + \sum_{j \le 2}2 F(\bx_j)u_j^3 + 
\sum_{k \le 9-r}2 F(\by_k)v_k^3 \Bigr| < 1.
\end{equation}
Substituting \eqref{bound} into \eqref{error} yields
\begin{equation} \label{lasterror}
F(\bz) = \sum_{i \le r-2} F(\bw_i)t_i^3 + \sum_{j \le 2} F(\bx_j)u_j^3 + 
\sum_{k \le 9-r} F(\by_k)v_k^3 + O(N^{\cE+\theta}),
\end{equation}
where $\cE= \max(3+3\alpha (9-r)-E_1, 3(2+\alpha (8-r) ) - \alp E_2) < 0$. Since $\theta$ is small we also have $\cE+\theta < 0$. Combining \eqref{small} and \eqref{lasterror} now yields $|F(\bz)| <1$. This implies that $F$ takes small values, for if $\bw_1, \ldots, \bw_{r-2}, \bx_1, \bx_2, \by_1, \ldots, \by_{9-r}$ were linearly dependent then by \eqref{lasterror} we could apply Lemma \ref{trick}.
\end{proof}

Choose $E_1$ close to as large as possible such that 
\[ (s-\hw_3^{(10-r)}(27-3r+\eps)) / (r-1) > \hw_3^{(2)}(E_1).\]
By \eqref{initial9} we now have $a_{r-1} > \hw_3^{(2)}(E_1)$. We need $s$ to be large enough to ensure that $E_1 > 3$. Choose $E_2 > 3(8-r)$ close to as small as possible such that $(E_1 -3)(E_2-8+r) > 18(9-r)$. In view of Lemma \ref{second9} we may assume that $a_r \le \hw_3^{(9-r)}(E_2)$, which implies that
\[ a_{r-1} \ge (s-\hw_3^{(9-r)}(E_2))/(r-1). \]
If $s$ is large enough then this bound will be better than our initial bound \eqref{initial9}, in which case we can repeat this procedure until a contradiction is reached. This would imply that $F$ takes small values.

For each $r$, the upper bound that we obtain for $s_0^{(r)}$ is not much greater than the theoretical limit of our approach, that being our upper bound for $\hw_3^{(10-r)}(27-3r+\eps)$. Little is lost, therefore, from the guesswork and computer optimisation involved here.

\textbf{Case: $r=5,6$.} Here we can achieve better bounds using eight quasi-diagonalising vectors. We use the following analogously-proven variants of Lemmas \ref{first9} and \ref{second9}. These rely on Theorem \ref{Bru8} instead of Theorem \ref{Bru9} as a key ingredient.

\begin{lemma} \label{first8}
Suppose $(r-1)a_{r-1} < s-\hw_3^{(9-r)}(\eps-3+45(9-r)/8)$. Then $F$ takes small values.
\end{lemma}

\begin{lemma} \label{second8} 
Let $E_1$ and $E_2$ be real numbers such $H_1 > 0$ and $H_1H_2 >8-r$, where $H_1 = (4E_1-33)/45$ and $H_2 = r-8+8(E_2+3)/45$. Suppose $a_{r-1} > \hw_3^{(2)}(E_1)$ and $a_r > \hw_3^{(8-r)}(E_2)$. Then $F$ takes small values.
\end{lemma}

The strategy in this case is much the same. Again the upper bound that we obtain for $s_0^{(r)}$ is not much greater than the theoretical limit of our approach, which is now our upper bound for $\hw_3^{(9-r)}(\eps-3+45(9-r)/8)$. 

\section{Implementation}
\label{finish}

\textbf{Case: $r=2$.} Assume for the sake of contradiction that $s \ge 120~897~257$ and $F$ does not take small values. By Lemma \ref{first9} we have 
\[ a_1 \ge 120~897~257 - 120~893~893 > \hw_3^{(2)}(14.6992), \]
so by Lemma \ref{second9} we have
\[ a_2 \le \hw_3^{(7)}(28.77) \le 120~847~458. \]
Now $a_1 \ge  49~799  >  \hw_3^{(2)}(42)  $, so by Lemma \ref{second9} we have
\[ a_2 \le \hw_3^{(7)}(21.2308) < 54~000~000      < s/2, \]
contradiction.

\textbf{Case: $r=3$.} Assume for the sake of contradiction that $s \ge 35~042~291$ and $F$ does not take small values. By Lemma \ref{first9} we have 
\[ a_2 \ge (35~042~291 - 35~037~484)/2 > \hw_3^{(2)}(12.705), \]
so by Lemma \ref{second9} we have
\[ a_3 \le \hw_3^{(6)}(26.1283) \le 34~956~075. \]
Now $a_2 \ge  43~108  >  \hw_3^{(2)}(40) $, so by Lemma \ref{second9} we have
\[ a_3 \le \hw_3^{(6)}(17.919) \le   13~000~000. \]
Now $a_2 > 10~000~000  >  \hw_3^{(2)}(267) $, so by Lemma \ref{second9} we have
\[ a_3 \le \hw_3^{(6)}(15.41) < 9~000~000   < s/3, \]
contradiction.

\textbf{Case: $r=4$.} Assume for the sake of contradiction that $s \ge 8~324~100$ and $F$ does not take small values. By Lemma \ref{first9} we have 
\[ a_3 \ge (8~324~100 - 8~319~167)/3 > \hw_3^{(2)}(10.6989), \]
so by Lemma \ref{second9} we have
\[ a_4 \le \hw_3^{(5)}(23.69) \le 8~300~761 . \]
Now $a_3 > 7~779   \ge  \hw_3^{(2)}(20.935) $, so by Lemma \ref{second9} we have
\[ a_4 \le \hw_3^{(5)}( 17.02 ) \le     3~532~167.\]
Now $a_3 \ge     1~597~311 > \hw_3^{(2)}(143)   $, so by Lemma \ref{second9} we have
\[ a_4 \le \hw_3^{(5)}( 12.7 ) < 2~000~000   < s/4, \]
contradiction.

\textbf{Case: $r=5$.} Assume for the sake of contradiction that $s \ge 1~164~774$ and $F$ does not take small values. By Lemma \ref{first8} we have 
\[ a_4 \ge (1~164~774-1~149~469)/4  > \hw_3^{(2)}(15.215), \]
so by Lemma \ref{second8} we have
\[ a_5 \le \hw_3^{(3)}(41.132) \le 1~148~061. \]
Now $a_4 > 4~178    >    \hw_3^{(2)}(16)$, so by Lemma \ref{second8} we have
\[ a_5 \le \hw_3^{(3)}(38.371) \le    950~987.\]
Now $a_4 >53~469 > \hw_3^{(2)}(43)$, so by Lemma \ref{second8} we have
\[ a_5 \le \hw_3^{(3)}(19.34)  <160~000 <  s/5, \]
contradiction.

\textbf{Case: $r=6$.} Assume for the sake of contradiction that $s \ge 77~027$ and $F$ does not take small values. By Lemma \ref{first8} we have 
\[ a_5 \ge (77~027 - 67~151)/5  > \hw_3^{(2)}(11.52), \]
so by Lemma \ref{second8} we have
\[ a_6 \le \hw_3^{(2)}(47) \le   66~301 . \]
Now $a_5 >2~145    >  \hw_3^{(2)}(12)   $, so by Lemma \ref{second8} we have
\[ a_6 \le \hw_3^{(2)}(42+\eps) \le 50~761.\]
Now $a_5 > 5~253   \ge  \hw_3^{(2)}(17.76 )   $, so by Lemma \ref{second8} we have
\[ a_6 \le \hw_3^{(2)}(21.56) \le 8~621 <s/6,\]
contradiction.

\bibliographystyle{amsbracket}
\providecommand{\bysame}{\leavevmode\hbox to3em{\hrulefill}\thinspace}

\end{document}